\documentclass{amsart}
\usepackage{latexsym,amsthm, amsmath,amssymb,amsfonts}
\newtheorem{thm}{Theorem}[section]
\newtheorem{cor}[thm]{Corollary}
\newtheorem{lem}[thm]{Lemma}

\theoremstyle{definition}\newtheorem{defn}[thm]{Definition}
\theoremstyle{remark}
\newtheorem{rem}[thm]{Remark}
\numberwithin{equation}{section}

\allowdisplaybreaks[1]
\begin{document}

\title[MULTIPLICATION CONDITIONAL EXPECTATION OPERATORS]
{Some properties of MCE OPERATORS between different ORLICZ SPACES}

\author{\sc\bf Y. Estaremi }
\address{\sc Y. Estaremi}
\email{yestaremi@pnu.ac.ir}
\address{ Department of Mathematics, Payame Noor University (PNU), P. O. Box: 19395-3697, Tehran- Iran}

\thanks{}

\thanks{}

\subjclass[2010]{ 47B38}

\keywords{multiplication operator, Conditional expectation, continuous operators, closed-range operators, finite-rank operators, Orlicz space.}

\date{}

\dedicatory{}

\commby{}

%%% ----------------------------------------------------------------------
\begin{abstract}
In this paper we study some basic properties, like boundedness and closedness of range, of multiplication conditional expectation(MCE) operators between different Orlicz spaces.
\noindent {}
\end{abstract}

\maketitle

\section{ \sc\bf Introduction }

 Our concern in this paper is to provide some necessary conditions, sufficient conditions, and  some simultaneously necessary and sufficient conditions for the multiplication conditional expectation or briefly MCE operators between distinct Orlicz spaces to be bounded or to have closed range or finite rank. Our results generalize and improve on some recent results to be found in the literature. One of the most important properties of MCE operators is that a large class of bounded operators on measurable functions spaces as well as on $L^p$- spaces are of the form of MCE operators. One can find many great papers that are about MCE operators on $L^p$- spaces. For instance the more important ones are \cite{dou,dhd,ej,g,lam,l,mo}. In addition, we investigated boundedness and compactness of MCE operators on Orlicz spaces in \cite{ye}. In this paper we continue our project to characterize closed range MCE operators between different Orlicz spaces.

\section{\sc\bf Preliminaries  and basic lemmas}

 In this section, for the convenience of the reader, we gather some essential facts on Orlicz spaces and prove two basic lemmas for later use. For more details on Orlicz spaces,  see \cite{kr,raor}.

A function $\Phi:\mathbb{R}\rightarrow [0,\infty]$ is called a \textit{Young function}  if $\Phi$ is   convex, even,  and  $\Phi(0)=0$; we will also assume that $\Phi$ is neither identically zero nor identically infinite on $(0,\infty)$. The fact that $\Phi(0)=0$, along with the convexity of $\Phi$, implies that $\lim_{x\rightarrow 0^+}\Phi(x)=0$; while  $\Phi\neq 0$, again along with the convexity of $\Phi$, implies that   $\lim_{x\rightarrow\infty}\Phi(x)=\infty$. We set
$a_{\Phi}:=\sup\{x\geq0:\Phi(x)=0\}$
 and
$
b_{\Phi}:=\sup\{x>0:\Phi(x)<\infty\}.
$
Then it can be checked that  $\Phi$ is continuous and nondecreasing on
$[0,b_{\Phi})$ and strictly increasing on $[a_{\Phi},b_{\Phi})$. We also assume the left-continuity of the function $\Phi$ at $b_\Phi$, i.e. $\lim_{x\rightarrow b_\Phi^-} \Phi(x)=\Phi(b_\Phi)$.

 To each Young  function $\Phi$ is  associated another
 convex function $\Psi:\mathbb{R}\rightarrow[0,\infty)$ with similar properties,  defined by
$$
\Phi^*(y)=\sup\{x|y|-\Phi(x):x\geq0\} \quad (y\in\mathbb{R}).
$$
The function $\Phi^*$ is called the  \textit{function  complementary} to $\Phi$ in the sense of Young. Any pair of complementary functions $(\Phi,\Phi^*)$ satisfies Young's inequality $xy\leq \Phi(x)+\Phi^*(y)\,\, (x,y\geq 0)$.

The generalized inverse of the Young function $\Phi$ is defined by
$$
\Phi^{-1}(y)=\inf \{ x\geq 0: \Phi(x)> y\} \quad (y\in [0,\infty)).
$$
Notice that if $x\geq0$, then $\Phi\big(\Phi^{-1}(x)\big)\leq x$,
and if $\Phi(x)<\infty$, we also have $x\leq\Phi^{-1}\big(\Phi(x)\big)$. There are equalities in either case when $\Phi$ is a Young function vanishing only at zero and taking only finite values.
Also, if $(\Phi,\Phi^*)$ is a pair of complementary
Young functions, then
\begin{equation}\label{12}
x<\Phi^{-1}(x)\Phi^{*^{-1}}(x)\leq 2x
\end{equation}
for all $x\geq0$ (Proposition 2.1.1(ii) \cite{raor}).

By an $N$-\textit{function} we mean a  Young function vanishing only at zero, taking only finite values, and such that $\lim_{x\rightarrow\infty}\Phi(x)/x=\infty$ and $\lim_{x\rightarrow 0^+}\Phi(x)/x=0$. Note that then $a_\Phi=0,$ $b_\Phi=\infty$, and, as we said above,   $\Phi$ is continuous and strictly increasing on $[0,\infty)$. Moreover, a function complementary to an $N$-function is again an $N$-function.

A Young function $\Phi$ is said to satisfy the
$\Delta_{2}$-condition at $\infty$ if $\Phi(2x)\leq
K\Phi(x) \; ( x\geq x_{0})$  for some constants
$K>0$ and $x_0>0$. A Young function $\Phi$ satisfies the
$\Delta_{2}$-condition globally if $\Phi(2x)\leq
K\Phi(x) \; ( x\geq 0)$  for some
$K>0$.

A Young function $\Phi$ is said to satisfy the
$\Delta'$-condition (respectively, the $\nabla'$-condition) at $\infty$, if there exist  $ c>0$
(respectively, $b>0$) and $x_0>0$ such that
$$
\Phi(xy)\leq c\,\Phi(x)\Phi(y) \quad (x,y\geq x_{0})
$$
$$
(\mbox{respectively, }  \Phi(bxy)\geq \Phi(x)\Phi(y) \quad ( x,y\geq x_{0})).
$$
If $x_{0}=0$, these conditions are said to hold
globally. Notice that if $\Phi\in \Delta'$, then  $\Phi\in
\Delta_{2}$ (both at $\infty$ and globally).

 Let $\Phi, \Psi$ be Young
functions. Then $\Phi$ is called stronger than $\Psi$ at $\infty$, which is denoted by $\Phi\mathrel{\overset{\makebox[0pt]{\mbox{\normalfont\scriptsize\sffamily $\ell$ }}}{\succ}}\Psi$ [or $\Psi\mathrel{\overset{\makebox[0pt]{\mbox{\normalfont\scriptsize\sffamily $\ell$}}}{\prec}}\Phi$], if
$$
\Psi(x)\leq\Phi(ax)\quad (x\geq x_0)
$$
for some $a\geq0$ and $x_0>0$; if $x_0=0$, this condition is
said to hold globally and is then denoted by $\Phi\mathrel{\overset{\makebox[0pt]{\mbox{\normalfont\scriptsize\sffamily $a$}}}{\succ}}\Psi$ [or $\Psi\mathrel{\overset{\makebox[0pt]{\mbox{\normalfont\scriptsize\sffamily $a$}}}{\prec}}\Phi$]. We record  the following observation for later use.

\begin{lem}\label{l2} If $\Phi, \Psi, \Theta$ are Young functions vanishing only at zero, taking only finite values, and such that
$$
\Phi(xy)\leq\Psi(x)+\Theta(y)\quad (x,y\geq0),
$$
then $\Psi\mathrel{\overset{\makebox[0pt]{\mbox{\normalfont\scriptsize\sffamily $\ell$}}}{\nprec}}\Phi$, and hence also  $\Psi\mathrel{\overset{\makebox[0pt]{\mbox{\normalfont\scriptsize\sffamily $a$}}}{\nprec}}\Phi$.
\end{lem}

Let $(X, \Sigma, \mu)$ be a complete $\sigma$-finite
 measure space and let $L^0(\Sigma)$  be the linear space of  equivalence classes of $\Sigma$-measurable
real-valued functions on $X$, that is, we identify functions equal $\mu$-almost everywhere on $X$. The support $S(f)$ of a
measurable function $f$ is defined by $S(f):=\{x\in X : f(x)\neq
0\}$. For a Young function  $\Phi$, the space
$$
L^{\Phi}(\Sigma)=\left\{f\in L^0(\Sigma):\exists k>0,
\int_X\Phi(kf)d\mu<\infty\right\}
$$
is a Banach space if it is equipped with the norm
$$
\|f\|_{\Phi}=\inf\left\{k>0:\int_X\Phi(f/k)d\mu\leq1\right\}.
$$
 The couple $(L^{\Phi}(\Sigma), \|\cdot\|_{\Phi})$ is called the Orlicz space generated by a Young function $\Phi$.
 Let $\Phi(x)=|x|^p/p$ with $1<p<\infty$; $\Phi$ is then a Young function and $\Psi(x)=|x|^{p'}/p'$, with $1/p+1/p'=1$, is the  Young function complementary to $\Phi$.
 Thus, with this function $\Phi$ we retrieve the classical Lebesgue space $L^p(\Sigma)$, i.e. $L^\Phi(\Sigma)=L^p(\Sigma)$.

Recall that an atom of the measure space $(X,\Sigma,\mu)$ is a set $A\in\Sigma$ with $\mu(A)>0$ such that if
$F\in\Sigma$ and $F\subset A$, then either $\mu(F)=0$ or
$\mu(F)=\mu(A)$. A measure space $(X,\Sigma,\mu)$ with no atoms is
called a non-atomic measure space. It is well-known  that if $(X, \Sigma, \mu)$ is a $\sigma$-finite measure space, then for every measurable real-valued function $f$ on $X$ and every atom $A$, there is a real number, denoted  by $f(A)$,  such that $f=f(A)\,\, \mu$-a.e. on $A$. Also, if $(X, \Sigma, \mu)$ is a $\sigma$-finite measure space that fails to be non-atomic, there is a non-empty countable set  of pairwise disjoint atoms $\{A_n\}_{n\in\mathbb{N}}$  with the property that ${B}:=X\setminus\bigcup_{n\in\mathbb{N}}A_n$ contains no atoms \cite{z}.\\
 Here we recall the next lemma that is a key tool in our investigations.
\begin{lem}\cite{ceh}\label{p1} Let $\Phi, \Psi$ be Young functions such that  $\Psi\mathrel{\overset{\makebox[0pt]{\mbox{\normalfont\scriptsize\sffamily $\ell$}}}{\nprec}}\Phi$. If $E$ is a
non-atomic $\Sigma$-measurable set with positive measure, then there exists $f\in L^{\Phi}(\Sigma)$ such that $f_{|_E}\notin L^{\Psi}(E)$.
\end{lem}
For a sub-$\sigma$-finite algebra $\mathcal{A}\subseteq\Sigma$, the
conditional expectation operator associated with $\mathcal{A}$ is
the mapping $f\rightarrow E^{\mathcal{A}}f$, defined for all
non-negative, measurable function $f$ as well as for all $f\in
L^1(\Sigma)$ and $f\in L^{\infty}(\Sigma)$, where
$E^{\mathcal{A}}f$, by the Radon-Nikodym theorem, is the unique
$\mathcal{A}$-measurable function satisfying
$$\int_{A}fd\mu=\int_{A}E^{\mathcal{A}}fd\mu, \ \ \ \forall A\in \mathcal{A} .$$
As an operator on $L^{1}({\Sigma})$ and $L^{\infty}(\Sigma)$,
$E^{\mathcal{A}}$ is idempotent and
$E^{\mathcal{A}}(L^{\infty}(\Sigma))=L^{\infty}(\mathcal{A})$ and
$E^{\mathcal{A}}(L^1(\Sigma))=L^1(\mathcal{A})$. Thus it can be
defined on all interpolation spaces of $L^1$ and $L^{\infty}$ such
as, Orlicz spaces \cite{besh}. If there is no possibility of
confusion, we write $E(f)$ in place of $E^{\mathcal{A}}(f)$. This
operator will play a major role in our work and we list here some
of its useful properties:

\vspace*{0.2cm} \noindent $\bullet$ \  If $g$ is
$\mathcal{A}$-measurable, then $E(fg)=E(f)g$.

\noindent $\bullet$ \ $\varphi(E(f))\leq E(\varphi(f))$, where
$\varphi$ is a convex function.

\noindent $\bullet$ \ If $f\geq 0$, then $E(f)\geq 0$; if $f>0$,
then $E(f)>0$.

\noindent $\bullet$ \ For each $f\geq 0$, $S(f)\subseteq S(E(f))$,
where  $S(f)=\{x\in X; f(x)\neq 0\}$.\\
A detailed discussion and verification of
most of these properties may be found in \cite{rao}.

Let $f\in L^{\Phi}(\Sigma)$.  It is not difficult to see that $\Phi(E(f))\leq
E(\Phi(f))$ and so by some elementary computations we get that $N_{\Phi}(E(f))\leq N_{\Phi}(f)$ i.e, $E$ is a
contraction on the Orlicz spaces.
As we defined in \cite{ye}, we say that the pair $(E, \Phi)$ satisfies the generalized conditional-type
H\"{o}lder-inequality (or briefly GCH-inequality) if there exists some positive constant $C$
such that for all $f\in L^{\Phi}(\Omega, \Sigma, \mu)$ and $g\in
L^{\Psi}(\Omega, \Sigma, \mu)$ we have
$$E(|fg|)\leq C \Phi^{-1}(E(\Phi(|f|)))\Phi^{*^{-1}}(E(\Phi^{*}(|g|))),$$
where $\Psi$ is the complementary Young function of $\Phi$. There are many examples of the pair $(E, \Phi)$ that satisfy GCH-inequality in \cite{ye}.

Finally in the following we give another key lemma that is important in our investigation. The proof is an easy exercise.
\begin{lem}\label{l1n} If $\Phi$ is a Young's function and $f$ is a $\Sigma$-measurable function such that $E(f)$ and $E(\Phi(f))$ are defined, then $S(E(f))=S(E(\Phi(f)))$.
\end{lem}
We keep the above notations throughout the paper.

\section{ \sc\bf  Boundedness of  the MCE operators}\label{trzy}

 In the section
 we state various  necessary conditions and sufficient conditions under which the MCE
 operator $EM_u$ between distinct Orlicz spaces is bounded. First we give a definition of MCE operators. 
 \begin{defn} Let $\Phi$, $\Psi$ be Young functions and $u:X\rightarrow \mathbb{C}$ be a measurable function on the measure space $(X,\Sigma,\mu)$. The multiplication conditional expectation operator (MCE operator) from $L^{\Phi}(\Sigma)$ into $L^{\Psi}(\Sigma)$ is defined by $EM_u(f)=E(uf)$ for every $f\in L^{\Phi}(\Sigma)$ such that $E(uf)\in L^{\Psi}(\Sigma)$.
 \end{defn}
 Here we find that there is no non-zero bounded MCE operator from $L^{\Phi}(\Sigma)$ into $L^{\Psi}(\mathcal{A})$,  when the measure space $(X,\mathcal{A},\mu)$ is non-atomic.
%\begin{pro}\label{t2} Let $\Phi$ and $\Psi$ be Young functions vanishing only at zero, taking only finite values, and  such that $\Phi$ is an N-function, %$\Phi\in \Delta'$, $\Psi\in \Delta_2$, and   $\Theta:=\Psi^*\circ\Phi^{{*}^{-1}}$ is  a Young function.  If $u\in L_+^0(\Sigma)$ and $EM_u$ is a bounded MCE %operator from $L^{\Phi}(\Sigma)$ into $L^{\Psi}(\mathcal{A})$, then $E(\Phi^{*}(\bar{u}))\in L^{\Theta^{*}}(\mathcal{A})$.
%
%\end{pro}

\begin{thm}\label{p0} Let $\Phi, \Psi$ be Young functions such that  $\Psi\mathrel{\overset{\makebox[0pt]{\mbox{\normalfont\scriptsize\sffamily $\ell$}}}{\nprec}}\Phi$ and let $(X,\mathcal{A},\mu)$ be  a non-atomic measure space. Then there are no  non-zero bounded MCE operators $EM_u$ from $L^{\Phi}(\Sigma)$ into $L^{\Psi}(\mathcal{A})$.
\end{thm}

\begin{proof} Suppose, to the contrary, that  $EM_{u}$ is a non-zero bounded linear operator from $L^{\Phi}(X)$ into $L^{\Psi}(X)$ and let
$$E_n=\left\{x\in X:E(u)(x)>\frac{1}{n}\right\}.$$
Then  $\{E_n\}_{n\in \mathbb{N}}$ is an increasing sequence of $\mathcal{A}$-measurable sets. Since $EM_{u}$ is non-zero, $\mu(E_m)>0$ for some $m\in \mathbb{N}$, whence also $\mu(E_n)>0$ for all $n\geq m$. We assume without loss of generality that $\mu(E_n)>0$ for all $n\in \mathbb{N}$. Let $F\subset E:=\bigcup_n E_n$ and $0<\mu(F)<\infty$. The assumption  that $\Psi\mathrel{\overset{\makebox[0pt]{\mbox{\normalfont\scriptsize\sffamily $\ell$}}}{\nprec}}\Phi$ implies that an increasing sequence of positive numbers $\{y_n\}$ can be found such that $\Psi(y_n)>\Phi(2^nn^3y_n)$. Since the measure space $(X,\mathcal{A},\mu)$ is non-atomic, we can choose a sequence $\{F_n\}$ of pairwise disjoint measurable subsets of $F$ such that $F_n\subset E_n$ and
$\mu(F_n)=\frac{\Phi(y_1)\mu(F)}{2^n\Phi(n^3y_n)}$. This is possible because $\Phi(n^3y_n)\geq \Phi(y_n)\geq \Phi(y_1)$, and so $\mu(F_n)\leq \frac{\mu(F)}{2^n}$, and $\sum^\infty_{n=1}\frac{\mu(F)}{2^n}=\mu(F)$.

 Define  the function $\displaystyle f:=\sum^\infty_{n=1}b_n\chi_{F_n}$, where $b_n:=n^2y_n$, and take  arbitrary  $\alpha>0$. Then for a natural number  $n_0>\alpha$ we have
 \begin{align*}
I_{\Phi}(\alpha f)&=\int_X\Phi(\alpha f)d\mu=\displaystyle\sum^{\infty}_{n=1}\int_X\Phi(\alpha\, b_n)\chi_{F_n} d\mu\\
&\leq\sum^{n_0}_{n=1}\Phi(\alpha\,b_n)\mu(F_n)+\mu(F)\sum_{n> n_0}\frac{\Phi(n^3y_n)\Phi( y_1)}{2^n\Phi(n^3 y_n)}<\infty.
\end{align*}
This implies that $f\in L^{\Phi}(\mathcal{A})$. But for  $m_0>0$ such that $\frac{1}{m_0}<\alpha$, we obtain
\begin{align*}
I_{\Psi}(\alpha EM_{u}f)&=\int_X\Psi(\alpha EM_{u}f)d\mu=\sum\limits_{n\geq0}\int_{F_n} \Psi(\alpha E(u)f)d\mu\\
&\geq\sum_{n\geq m_0}\frac{1}{n}2^n\Phi(n^3 y_n)\mu(F_n)\geq\mu(F)\sum_{n\geq m_0}\frac{1}{n}\Phi(y_1)=\infty,
\end{align*}
which contradicts the boundedness of $EM_u$. (In fact, we even proved that $EM_u$ does not act from $L^{\Phi}(\Sigma)$ into $L^{\Psi}(\mathcal{A})$).
\end{proof}
Now we provide some necessary conditions for the boundedness of the MCE
operator $EM_u$ from $L^{\Phi}(\Sigma)$ into $L^{\Psi}(\Sigma)$ in the case $\Phi(xy)\leq\Psi(x)+\Theta(y)$
for some Young function $\Theta$ and for all $x,y\geq 0$.

\begin{thm}\label{t3}  Let $\Phi,\Psi, \Theta$ be Young functions as are defined in the first section, and such
that $\Phi(xy)\leq\Psi(x)+\Theta(y)$ for all $x,y\geq 0$. If $u\in L^0(\Sigma)$  induces a bounded MCE
operator $EM_u:L^{\Phi}(\Sigma)\rightarrow L^{\Psi}(\Sigma)$, then
\begin{enumerate}

\item[(i)] $E(u)=0$, \ $\mu$-a.e. on $B$, the non-atomic part of $X$;

\item[(ii)] $\sup\limits_{n\in \mathbb{N}}E(u)(A_n)\,\Theta^{-1}(\frac{1}{\mu(A_n)})<\infty$.
\end{enumerate}
\end{thm}

\begin{proof}
  Suppose that $EM_u$ is bounded. First we prove (i). If $\mu\{x\in B:E(u)(x)\neq0\}>0$,
  then we can find a constant $\delta>0$ such that for $F=\{x\in B:E(u)(x)>\delta\}$  we have $\mu(F)>0$. Since $F$ is non-atomic $\mathcal{A}$-measurable set, $\mu(F)>0$, and $\Phi(xy)\leq\Psi(x)+\Theta(y)$ for all $x,y\geq 0$, by Lemmas \ref{l2} and \ref{p1} we have that there exists $f\in L^{\Phi}(\mathcal{A})$
   such that $f_{|_F}\notin L^{\Psi}(F, \mathcal{A}_F)$, and so
\begin{align*}
\infty&=\int_F\Psi\left(\frac{\delta f}{\| EM_uf\|_{\Psi}}\right)d\mu\\
&\leq\int_X\Psi\left(\frac{E(u)f}{\| EM_uf\|_{\Psi}}\right)d\mu=\int_X\Psi\left(\frac{E(uf)}{\| EM_uf\|_{\Psi}}\right)d\mu\leq1,
\end{align*}
which is a contradiction. Thus (i) holds.

Now we prove (ii). We may assume that the function $u$ is not identically zero. For each $n\in \mathbb{N}$, put $f_n=\Phi^{-1}(\frac{1}{\mu(A_n)})\chi_{A_n}$. It is clear that $f_n\in L^{\Phi}(\Sigma)$ and $I_{\Phi}(f_n)=1$, whence $\|f_n\|_{\Phi}=1$. Since the operator $M_u$ is bounded, we have
\begin{align*}
1\geq\int_X\Psi\left(\frac{E(u.f_n)}{\| E(u.f_n)\|_{\Psi}}\right)d\mu &=\int_{A_n}\Psi\left(\frac{E(u)\Phi^{-1}(\frac{1}{\mu(A_n)})}{\| E(u.f_n)\|_{\Psi}}\right)d\mu\\
&=\Psi\left(\frac{E(u)(A_n)\Phi^{-1}(\frac{1}{\mu(A_n)})}{\| E(u.f_n)\|_{\Psi}}\right)\mu(A_n).
\end{align*}
Therefore
\begin{align}\label{t34p}
\frac{E(u)(A_n)\Phi^{-1}(\frac{1}{\mu(A_n)})}{\Psi^{-1}(1/\mu(A_n))}\leq\|E(u.f_n)\|_{\Psi}.
\end{align}
Plug $x=1/\mu(A_n)$ in the inequality
$\Psi^{-1}(x)\Theta^{-1}(x)\leq 2\Phi^{-1}(x)$
derived in lines 2-3 of the proof of Lemma \ref{l2} and use inequality (\ref{t34p}) to obtain

\begin{align*}
E(u)(A_n)\,\Theta^{-1}(\frac{1}{\mu(A_n)})&\leq E(u)(A_n)\frac{2\Phi^{-1}(1/\mu(A_n))}{\Psi^{-1}(1/\mu(A_n))}\leq 2\| E(u.f_n)\|_{\Psi}\leq 2\|EM_u\|<\infty.
\end{align*}
This completes the proof.
\end{proof}
By using the Lemma \ref{l1n}  we have the next straightforward consequence.
\begin{cor} Under the assumptions of Theorem \ref{t3}, if $(X, \Sigma, \mu)$
is a non-atomic measure space, then the MCE operator $EM_u$ is bounded from
$L^{\Phi}(\Sigma)$ into $L^{\Psi}(\Sigma)$ if and only if $EM_u=0$.
\end{cor}
\begin{proof}
Since $S(E(uf))\subseteq S(E(u))=S(E(\Psi(u)))$, for each $f\in L^{\Phi}(\Sigma)$, then we get the result.
\end{proof}
Now we present some sufficient conditions of the continuity of the operator $EM_u$ from one Orlicz space into another.
\begin{thm}\label{t44}
Let $\Phi, \Psi$ be Young functions such that
  $\Phi,\Psi\in\Delta'$,  $\Psi\circ\Phi^{-1}$ is a Young function and GCH inequality holds for the pair $(\Phi, \Phi^*)$.
  Then for $u\in L^0(\Sigma)$ the MCE operator $EM_u:L^{\Phi}(\Sigma)\rightarrow L^{\Psi}(\mathcal{A})$  if
\begin{enumerate}
\item[(i)] $E(\Phi^*(u))=0$  $\mu$-a.e. on $B$;\\
\item[(ii)] $M=\sup\limits_{n\in \mathbb{N}}\Psi\left[\frac{Cc_1\Phi^{*^{-1}}(E(\Phi^*(u)))(A_n)}
{\Phi^{-1}(\mu(A_n))}\right]\mu(A_n)<\infty$,
in which $C$ comes from the GCH- inequality and $c_1$ comes from the fact that $\Phi\in\Delta'$.
\end{enumerate}
\end{thm}
\begin{proof}
 Suppose that (i) and (ii) hold and set  $M:=\sup\limits_{n\in \mathbb{N}}\Psi\left(Cc_1\frac{\Phi^{*^{-1}}(E(\Phi^{*}(u)))(A_n)}{\Phi^{-1}(\mu(A_n))}\right)\mu(A_n) $. Then for each $f\in L^{\Phi}(\Sigma)$ with $\|f\|_{\Phi}\leq1$ we have
\begin{align*}
I_{\Phi_2}(EM_uf)&=\int_X\Psi(EM_uf)d\mu\\
&\leq\int_X \Psi(C\Phi^{*^{-1}}(E(\Phi^{*}(u)))\Phi^{-1}(E(\Phi(f))))d\mu\\
&=\int_{\bigcup_n A_n}\Psi(C\Phi^{*^{-1}}(E(\Phi^{*}(u)))\Phi^{-1}(E(\Phi(f))))d\mu\\
&=\sum\limits_{n\in \mathbb{N}}\Psi(C\Phi^{*^{-1}}(E(\Phi^{*}(u)))\Phi^{-1}(E(\Phi(f))))(A_n)\mu(A_n)\\
&=\sum\limits_{n\in \mathbb{N}}\Psi\left(C\frac{\Phi^{*^{-1}}(E(\Phi^{*}(u)))(A_n)}{\Phi^{-1}(\mu(A_n))}c_1\Phi^{-1}(\mu(A_n)E(\Phi(f))(A_n))\right)\mu(A_n)\\
&\leq c_2\sum\limits_{n\in \mathbb{N}}\Psi\left(Cc_1\frac{\Phi^{*^{-1}}(E(\Phi^{*}(u)))(A_n)}{\Phi^{-1}(\mu(A_n))}\right)\mu(A_n)\Psi\circ\Phi^{-1}(\mu(A_n)E(\Phi(f))(A_n))\\
&\leq c_2M\sum\limits_{n\in \mathbb{N}}\Psi\circ\Phi^{-1}(\mu(A_n)E(\Phi(f))(A_n))\\
&\leq c_2M\Psi\circ\Phi^{-1}(\sum\limits_{n\in \mathbb{N}}\mu(A_n)E(\Phi(f))(A_n))\\
&\leq c_2M\Psi\circ\Phi^{-1}(1),
\end{align*}
where $C,c_1,c_2$ come from from GCH-inequality and the fact that $\Phi, \Psi\in\Delta'$. Also we used the superadditivity of the convex function  $\Psi\circ\Phi^{-1}$ on the interval $[0,\infty)$. Therefore we have $\|EM_u f\|_{\Psi}\leq \max(c_2M\,\Psi\circ\Phi^{-1}(1),1)$, and so $EM_u$ is bounded.
\end{proof}
\begin{thm}\label{t36}
   Let $\Phi$ and $\Psi$ be Young functions an $EM_{u}:\mathcal{D}\subseteq L^{\Phi}(\Sigma)\rightarrow L^{\Psi}(\Sigma)$ be well defined. Then the followings hold:

 (i) Suppose that there exists a Young function $\Theta$ such that $\Phi^{-1}(x)\Theta^{-1}(x)\leq \Psi^{-1}(x)$ for $x\geq0$ and $(E,\Phi)$ satisfies GCH-inequality. In this case if $\Phi^{*^{-1}}(E(\Phi^*(u)))\in L^{\Theta}(\mathcal{A})$, then $EM_u$ from $L^{\Phi}(\Sigma)$ into $L^{\Psi}(\Sigma)$ is bounded.

 (ii) Let $\Theta=\Psi^*\circ\Phi^{*^{-1}}$ be a Young function, $\Theta\in \bigtriangleup_2$ and $\Phi^* \in \bigtriangleup_2$. In this case if $EM_u$ is bounded from $L^{\Phi}(\Sigma)$ into $L^{\Psi}(\Sigma)$, then $E(\Phi^*(\bar{u}))\in L^{\Theta^*}(\mathcal{A})$. Consequently $\Phi^{*^{-1}}(E(\Phi^*(\bar{u})))\in L^{\Theta^*\circ\Phi^*}(\mathcal{A})$.
\end{thm}
 \begin{proof} (i) Let $f\in L^{\Phi}(\Sigma)$ such that $\|f\|_{\Phi}\leq1$. This means that
 $$\int_{X}\Phi(\Phi^{-1}(E(\Phi(f))))d\mu=\int_{X}\Phi(f)d\mu\leq1$$
 and so $\|(\Phi^{-1}(E(\Phi(f))))\|_{\Phi}\leq1$. Therefore by using GCH-inequality we have
 \begin{align*}
 \|(E(uf))\|_{\Psi}&\leq \|(\Phi^{-1}(E(\Phi(f))))\|_{\Phi}\|(\Phi^{*^{-1}}(E(\Phi^*(u))))\|_{\Theta}\\
 &\leq \|(\Phi^{*^{-1}}(E(\Phi^*(u))))\|_{\Theta}.
 \end{align*}
Thus for all $f\in L^{\Phi}(\Sigma)$ we have $$\|(E(uf))\|_{\Psi}\leq \|f\|_{\Phi}\|\Phi{*^{-1}}E(\Phi^*(u)\|_{\Theta}.$$ And so the operator $EM_u$ is bounded.\\

(ii) Suppose that $EM_u$ is bounded. So the adjoint operator $M_{\bar{u}}=(EM_u)^{\ast}:L^{\Psi^*}(\mathcal{A})\rightarrow L^{\Phi^*}(\Sigma)$ is bounded. For $f\in L^{\Theta}(\mathcal{A})$ we have $\Phi^{*^{-1}}(f)\in L^{\Psi^*}(\mathcal{A})$. Consequently we get that
\begin{align*}
\int_{X}E(\Phi^*(\bar{u}))fd\mu&=\int_{X}\Phi^*(\bar{u})\Phi^*(\Phi^{*^{-1}}(f))d\mu\\
&\leq b\int_{X}\Phi^*(\bar{u}\Phi^{*^{-1}}(f))d\mu\\
&=b\int_{X}\Phi^*(M_{\bar{u}}(\Phi^{*^{-1}}(f)))d\mu<\infty.
\end{align*}
Therefore $\int_{X}E(\Phi^*(\bar{u}))fd\mu<\infty$ for all $f\in L^{\Theta}(\mathcal{A})$. This implies that $E(\Phi^*(\bar{u}))\in L^{\Theta^*}(\mathcal{A})$. This completes the proof.
\end{proof}
\begin{rem} The following results on the classical Lebesgue spaces, presented in \cite{ej}, are an immediate consequence of Theorems \ref{t44}, \ref{t3} and \ref{t36}.

\begin{itemize}\item[(1)] Taking $\Phi(x)=|x|^p/p$ and $\Psi(x)=|x|^q/q$, where $1<p<q<\infty$ and $\frac{1}{p}+\frac{1}{p'}=1$, $\frac{1}{q}+\frac{1}{q'}=1$, by  we obtain that the operator $EM_{u}$ induced by a function $u\in L_+^0(\Sigma)$ is bounded from $L^{p}(\Sigma)$ into $L^{q}(\Sigma)$ if and only if the following conditions hold:
\begin{enumerate}
\item[(i)] $E(u^{p'})(x)=0$ for $\mu$-almost all $x\in B$;

\item[(ii)] $\sup_{n\in \mathbb{N}}\frac{(E(u^{p'})(A_n))^{\frac{q}{p'}}}{(\mu(A_n))^{\frac{q}{p}}}<\infty$, where $q^{-1}+r^{-1}=p^{-1}$.
\end{enumerate}

\item[(2)] Similarly, taking $\Phi(x)=|x|^p/p$ and $\Psi(x)=|x|^q/q$, where $1<q<p<\infty$. Let $EM_{u}:\mathcal{D}\subseteq L^{p}(\Sigma)\rightarrow L^{q}(\Sigma)$ be well defined. Then the operator $EM_{u}$ from $L^{p}(\Sigma)$ into $L^{q}(\Sigma)$, where $1<q<p<\infty$, is bounded if and only if $(E(|u|^{p'}))^{\frac{1}{p'}}\in L^{r}(\mathcal{A})$, where $r=\frac{pq}{p-q}$.\\
    Specially, if $\mathcal{A}=\Sigma$, then $E=I$ and so the multiplication operator $M_u$ from $L^{p}(\Sigma)$ into $L^{q}(\Sigma)$ is bounded if and only if $u\in L^{r}(\Sigma)$.
\end{itemize}
\end{rem}

\section{ \sc\bf   MCE operators with closed-range and/or finite rank }

In this section we are going to investigate closed-range MCE operators between distinct Orlicz spaces.

 First we characterize closed-range MCE operators $EM_u:L^{\Phi}(\Sigma)\rightarrow L^{\Psi}(\Sigma)$ under the assumption $\Psi(xy)\leq\Phi(x)+\Theta(y)$  for all $x, y\geq 0$.
\begin{thm}\label{t33} Let $\Phi, \Psi, \Theta$ be  Young functions vanishing only at zero, taking only finite values, and such that  $\Psi,\Theta\in \Delta_2$ and $\Psi(xy)\leq\Phi(x)+\Theta(y)$  for all $x,y\geq 0$. If $u\in L_+^{\Theta}(\Sigma)$, then  the MCE operator $EM_u$ is a bounded operator from $L^{\Phi}(\Sigma)$ into $L^{\Psi}(\mathcal{A})$ and the following assertions are equivalent:
\begin{enumerate}

\item[(a)]$E(\Phi^{*}(u))=0$ a.e. on $B$ and the set $E=\{n\in \mathbb{N}: E(\Phi^{*}(u))(A_n)\neq 0\}$ is finite.
\item[(b)]$EM_u$ has finite rank.
\item[(c)]$EM_u$ has closed range.
\end{enumerate}
\end{thm}
\begin{proof} Let $S$ be the support of $E(\Phi^{*}(u))$. By GCH inequality we get that $S(E(uf))\subseteq S$ for all $f\in L^{\Phi}(\Sigma)$.  We may assume that $\mu(S)>0$, since otherwise $EM_u$ is a zero operator and there is nothing to prove.

We prove the implication (a) $\Rightarrow$ (b). Assume that (a) holds. Hence there is $r\in \mathbb{N}$ such that
$$
S=\bigcup\limits_{n\in E}{\frak A}_n=A_{n_1}\cup\ldots\cup A_{n_r}.
$$
Since $\Psi \in \Delta_2$, the set $\{\chi_{A_{n_1}},\ldots,\chi_{A_{n_r}}\}$ of characteristic functions generates the subspace
\begin{center}
$\{g\in L^{\Psi}(X,\mathcal{A}): g(x)=0$  for $\mu$-a.e. $x\in X\setminus S\}\cong L^{\Psi}(S,\mathcal{A}_{S})$.
\end{center}
The range of $EM_u$ is contained in the $r$-dimensional subspace $L^{\Psi}(\mathcal{A}_S)$, hence $EM_u$ has finite rank.

 (b) $\Rightarrow$ (c). If the range of the operator $EM_u$ in  $L^{\Psi}(\mathcal{A})$ is finite-dimensional, then it is also closed, as any finite-dimensional subspace of a Banach space is a closed subspace of this space.

 (c) $\Rightarrow$ (a). Let $EM_u$ have closed range and assume that $\mu\{x\in B:E(\Phi^{*}(u))(x)\neq0\}>0$. Then there is  $\delta>0$ such that the set $G=\{x\in B: E(\Phi^{*}(u))(x)\geq\delta\}$ has positive measure. Since $G$ is an $\mathcal{A}$- measurable set, then $EM_u(L^{\Phi}(G))\subseteq L^{\Psi}(G)$. It is easy to see that the restriction $E(u)_{|_G}$ induces a bounded MCE operator $EM_{u_{|_{G}}}=EM_{E(u)_{|_{G}}}$ from $L^{\Phi}(G,\mathcal{A}_G,\mu_{\mathcal{A}_G})$ into $L^{\Psi}(G,\mathcal{A}_G,\mu_{\mathcal{A}_G})$, and that if $EM_u$ has closed range, then  $EM_{u_{|_{G}}}$ has closed range as well.

 We will show that $EM_{u_{|_{G}}}(L^{\Phi}(G))=L^{\Psi}(G)$. Let $A$ be any $\mathcal{A}$-measurable subset of $G$ with $\mu(A)<\infty$, and define  the function $f_A:=\frac{1}{E(u_{|_{G}})}\chi_{A}$ in which $E(u_{|_{G}})=E(u)_{|_{G}}$.  We get
$$
I_{\Phi}(f_A)=\int_G \Phi\circ f_A\,d\mu=\int_A \Phi\circ\frac{1}{u_{|_{G}}}\,d\mu\leq\Phi(1/\delta)\mu(A)<\infty,
$$
  and so $f_A\in L^{\Phi}(G)$. Moreover,  $EM_{u_{|_{G}}} f_A =\chi_{A}$, which implies that the linear space $EM_{u_{|_{G}}}(L^{\Phi}(G))$ contains the set $\mathcal{F}$ of all linear combinations of characteristic functions of measurable subsets of $G$ with positive and finite measure.

  Now $\mathcal{F}$ is a dense subset of $L^{\Psi}(G)$ and, by assumption, $EM_{u_{|_{G}}}(L^{\Phi}(G))$ is a closed subspace of $L^{\Psi}(G)$, therefore $$L^{\Psi}(G)=\overline{\mathcal{F}}\subset EM_{u_{|_{G}}}(L^{\Phi}(G))\subset L^{\Psi}(G),$$
and so $EM_{u_{|_{G}}}(L^{\Phi}(G))= L^{\Psi}(G),$ as claimed.

  Consequently, we can define the inverse MCE operator
$$
EM_{\frac{1}{E(u_{|G})}}:L^{\Psi}(G)\rightarrow L^{\Phi}(G),\quad EM_{\frac{1}{E(u_{|G})}} f:=\frac{1}{E(u_{|_{G}})}f.
$$
The operator $EM_{\frac{1}{E(u_{|G})}}$ is bounded and, since $\Psi(xy)\leq\Phi(x)+\Theta(y)$\,\, ($x,y\geq 0$), we can apply Theorem \ref{t3} to conclude that $\frac{1}{E(u)}=0$  $\mu$-a.e. on $G$, which is absurd. This contradiction shows that $E(u)=0$  $\mu$-a.e. on $B$.

Next we show that the set
 $$E=\{n\in \mathbb{N}:E(\Phi^{*}(u))(A_n)\neq 0\}=\{n\in \mathbb{N}:E(u)(A_n)\neq 0\}$$
  is finite if $EM_u$ has closed range. If $E=\emptyset$, we have nothing to prove. So let us assume that $E\neq \emptyset$. Define $S=\bigcup\limits_{n\in E}A_n$.

Analogously as above, we can show that $EM_{u_{|_{S}}}(L^{\Phi}(S))=L^{\Psi}(S)$. Indeed, let $A$ be an $\mathcal{A}$-measurable subset of $S$ with $\mu(A)<\infty$. Define the function $f_A:=\frac{1}{E(u)_{|_{S}}}\chi_{A}$. The set $A$ having finite measure, we get
$$
I_{\Phi}(f_A)=\int_A \Phi\circ\frac{1}{E(u)_{|_{S}}}\,d\mu=\sum_{A_n\subset A}\Phi(1/E(u)(A_n))\mu(A_n)<\infty.
$$
Hence $f_A\in L^{\Phi}(S)$. Since  $EM_{u_{|_{S}}} f_A =\chi_{A}$, we conclude that $EM_{u_{|_{S}}}(L^{\Phi}(S))$ contains the set $\ell_f^0$ of all linear combinations of characteristic functions of subsets of $S$ with positive and finite measure.

Now $\ell_f^0$ is a dense subset of $L^{\Psi}(S)$, and $EM_{u_{|_{S}}}(L^{\Phi}(S))$ is a closed subspace of $L^{\Psi}(S)$, which implies that
$$L^{\Psi}(S)=\overline{\ell_f^0}\subset EM_{u_{|_{S}}}(L^{\Phi}(S))\subset L^{\Psi}(S),$$
and so $EM_{u_{|_{S}}}(L^{\Phi}(S))= L^{\Psi}(S)$.

We can thus define a bounded MCE operator $EM_{\frac{1}{u_{|_S}}}$ from $ L^{\Psi}(S, \mathcal{A}_S,\mu_{\mathcal{A}_S})$ into $L^{\Phi}(S, \mathcal{A}_S,\mu_{\mathcal{A}_S})$. Applying Theorem \ref{t3} to the operator $EM_{\frac{1}{u_{|_S}}}$, we obtain
$$
 \sup\limits_{n\in E} \frac{1}{E(u)(A_n)}\,\Theta^{-1}(\frac{1}{\mu(A_n)})< \infty.
 $$

Let $C=\sup\limits_{n\in E} \frac{1}{E(u)(A_n)}\,\Theta^{-1}(\frac{1}{\mu(A_n)})>0$. Since $E\neq\varnothing$ and $1\leq \Theta(Cu(A_n))\,\mu(A_n)$  for all $n\in E$, we have
 \begin{align*}
 \sum_{n\in E}1 \leq\sum_{n\in E}\Theta (CE(u)(A_n))\,\mu(A_n)&=\sum_{n\in E}\int_{A_n}\Theta\circ CE(u)\,d\mu\leq\int_{X}\Theta\circ Cu\,d\mu<\infty,
 \end{align*}
where the final inequality follows from the assumption that $\Theta\in \Delta_2$. Thus $E$ must be finite.
\end{proof}

In the next theorem we characterize closed-range MCE operators $EM_u:L^{\Phi}(\Sigma)\rightarrow L^{\Psi}(\Sigma)$ under the condition  that $\Phi(xy)\leq\Psi(x)+\Theta(y)$  for all $x, y\geq 0$.

\begin{thm}\label{t34} Let $\Phi, \Psi, \Theta$ be Young functions vanishing only at zero, taking only finite values, and such that  $\Psi,\Theta\in\Delta_2$ and $\Phi(xy)\leq\Psi(x)+\Theta(y)$ for all $x,y\geq 0$. If $EM_u$ is a bounded multiplication operator from $L^{\Phi}(\Sigma)$ into $L^{\Psi}(\Sigma)$ and  $\frac{1}{E(u)}\in L_+^{\Theta}(\Sigma)$,  then the following statements  are equivalent:
\begin{enumerate}

\item[(a)]The set $E=\{n\in \mathbb{N}: E(u)(A_n)\neq 0\}$ is finite.

\item[(b)]$EM_u$ has finite rank.

\item[(c)]$EM_u$ has closed range.
\end{enumerate}
\end{thm}
\begin{proof}
By Theorem \ref{t3}, we have that $E(u)=0$ on $B$. The proofs of implications (a) $\Rightarrow$ (b) and (b) $\Rightarrow$ (c) are as in the proof of  Theorem \ref{t33}.

 We prove the implication (c) $\Rightarrow$ (a). Using the  same notation as in the proof of Theorem \ref{t33}, let $S=\bigcup\limits_{n\in E}A_n$ and $E\neq\varnothing$. Since  $EM_u:L^{\Phi}(S)\rightarrow L^{\Psi}(S)$ is bounded, by Theorem \ref{t3}, we have
$$
\sup\limits_{n\in E} E(u)(A_n)\,\Theta^{-1}\left(\frac{1}{\mu(A_n)}\right)< \infty
$$
	
	Let $C=\sup\limits_{n\in E} E(u)(A_n)\,\Theta^{-1}(\frac{1}{\mu(A_n)})>0$. Since $E\neq\varnothing$, $1\leq \Theta\left(\frac{C}{E(u)(A_n)}\right)\,\mu(A_n)$ for all $n\in E$, and $\Theta \in \Delta_2$, we may write
	\begin{align*}
	\sum_{n\in E}1\leq\sum_{n\in E}\Theta\left(\frac{C}{E(u)(A_n)}\right)\,\mu(A_n)&=\sum_{n\in E}\int_{A_n}\Theta\circ \frac{C}{E(u)}\,d\mu\leq\int_{X} \Theta\circ\frac{C}{E(u)}\,d\mu<\infty.
	\end{align*}
	Thus $E$ is finite.
\end{proof}
\begin{rem}
As an applications of our results we derive characterizations of bounded and closed-range MCE operators in the special case of $L^p$-spaces.

\begin{itemize}
\item[(1)] If for $1<p<q<\infty$ and  the MCE operator $EM_{u}$ is bounded from $L^{p}(\Sigma)$ into $L^{q}(\Sigma)$ and $\frac{1}{E(u)}\in L^{\frac{pq}{p-q}}(\mathcal{A})$, then the following assertions are equivalent:

\begin{enumerate}
\item[(a)] $EM_u$ has closed range.

\item[(b)] $EM_u$ has finite rank.

\item[(c)] The set $\{n\in \mathbb{N}: E(u)(A_n)\neq0\}$ is finite.

\end{enumerate}

\item[(2)] If for $1<q<p<\infty$ the multiplication operator $M_{u}$ is bounded from $L^{p}(\Sigma)$ into $L^{q}(\Sigma)$, then the following assertions are equivalent:

\begin{enumerate}
\item[(a)]  $EM_u$ has closed range.

\item[(b)] $EM_u$ has finite rank.

\item[(c)]  $(u^{p'})(x)=0$ for $\mu$-almost all $x\in B$, and the set $\{n\in \mathbb{N}:E(u^{p'})(A_n)\neq0\}$ is finite.\\
\end{enumerate}
\end{itemize}
\end{rem}

%\textbf{Acknowledgement.}

\end{document}